\documentclass[10pt]{amsart}
\usepackage{amsmath}
\usepackage{amssymb}
\usepackage{graphicx, subfig}
\usepackage{a4wide}
\usepackage{pifont}

\newcommand{\Z}{\ensuremath{\mathbb{Z}}}
\newcommand{\R}{\ensuremath{\mathbb{R}}}

\newcommand{\M}{\ensuremath{{\mathcal M}}}

\renewcommand{\rho}{\varrho}
\renewcommand{\epsilon}{\varepsilon}

\newtheorem{thm}{Theorem}[section]

\newtheorem{cor}[thm]{Corollary}

\newtheorem{ex}[thm]{Illustration}

\begin{document}

\title{Bravais Colorings of $N$-fold Tilings}

\author{E.P. Bugarin}
\address{Mathematics department, Ateneo de Manila University, Loyola Heights, Quezon City 1108, Philippines}
\email{paobugs@yahoo.com}

\author{M.L.A. de las Pe\~nas}
\address{Mathematics department, Ateneo de Manila University, Loyola Heights, Quezon City 1108, Philippines}
\email{mlp@math.admu.edu.ph}

\begin{abstract}
In this work, a theory of color symmetry is presented that extends the ideas of traditional theories of color symmetry for periodic crystals to apply to non-periodic crystals. The color symmetries are associated to each of the crystalline sites and may correspond to different chemical species, various orientations of magnetic moments and colorings of a non-periodic tiling. In particular, we study the color symmetries of periodic and non-periodic structures via Bravais colorings of planar modules that emerge as the ring of integers in cyclotomic fields with class number one. Using an approach involving matrices, we arrive at necessary and sufficient conditions for determining the color symmetry groups and color fixing groups of the Bravais colorings associated with the modules $M_n = \Z[\exp(2\pi i/n)]$, and list the findings for $\M_{15} = \Z[\exp(2\pi i/15)]$ and $\M_{16} = \Z[\exp(\pi i/8)]$. In the second part of the paper, we discuss magnetic point groups of crystal and quasicrystal structures and give some examples of structures whose magnetic point group symmetries are described by Bravais colorings of planar modules.
\end{abstract}

\maketitle

\section{\bf Introduction} \label{intro}

The discovery of alloys with long-range orientational order and sharp diffraction images of non-crystallographic symmetries has initiated an intensive investigation of possible structures and physical properties of such systems. It was this amplified interest that established a new branch of solid state physics and also of discrete geometry called the theory of quasicrystals. To this day, since its discovery in 1984, quasicrystals have been studied intensively by metallurgists, physicists and mathematicians. Mathematicians in particular are increasingly intrigued by the underlying mathematical principles that govern quasicrystals, the geometry of diffraction patterns and their link to the study of the fascinating properties of non-periodic tilings. The most important manifestation of quasicrystals is their implicit long-range internal order that makes itself apparent in the diffraction patterns associated with them. The mathematics used to model such objects and to study their diffractive and their self-similar internal structures turns out to be highly interdisciplinary and includes algebraic number theory, theory of lattices, linear algebra, Fourier analysis, tiling theory, the study of self-similar structures and fractal measures, and dynamical systems. 

In this paper, a direct contribution to the theory of quasicrystals is provided from the mathematicianÕs perspective by addressing color symmetries of these structures. Color symmetries of quasicrystals continue to attract a lot of attention because not much is known about their classification. Particularly interesting are planar cases because, on the one hand, they show up in quasicrystalline $T$-phases, and, on the other hand, they are linked to the rather interesting classification of planar Bravais classes with $n$-fold symmetry (Mermin et al. 1986).

The planar Bravais classes are unique for the following 29 choices of n,
\begin{eqnarray} \label{eq:cls1}
 n=3,4,5,7,8,9,11,12,13,15,16,17,19,20, 21,24, \nonumber \\ 25,27,28,32,33,
           35,36,40,44,45,48,60,84.
\end{eqnarray}
The symmetry cases are grouped into classes with equal value of the Euler $\varphi$-function $$\varphi(n) = \#\{ 1 \leq k \leq n | \gcd(k, n) = 1\}.$$

The properties of quasicrystals are studied by looking at the structure of planar modules through their color symmetries that appear in quasicrystalline $T$-phases. This is done by studying the color symmetries associated with the cyclotomic integers $\M_n = \Z[\xi_n]$, the ring of polynomials in $\xi_n = \exp(2\pi i/n)$, a primitive $n$th root of unity for $n\in\{5, 8, 10, 12\}$, which yield the quasicrystallographic cases. 

The values of $n$ in Eq. (\ref{eq:cls1}) correspond to all cases where $\M_n = \Z[\xi_n]$ is a principal ideal domain and thus has class number one. If $n$ is odd, we have $M_n = M_{2n}$ and $\M_n$ thus has $2n$-fold symmetry. To avoid duplication of results, values of $n \equiv 2 \mod 4$ do not appear in Eq. (\ref{eq:cls1}). In this work we will study the color symmetries of $\M_n$ for all these values of $n$, which allow us to investigate aside from quasicrystallographic structures ($\varphi(n) = 4$), non-periodic structures in general. Two crystallographic cases ($\varphi(n) = 2$) are also covered, namely, the triangular lattice ($n = 3$) and the square lattice ($n = 4$).

The problem of color symmetry has been studied in much detail by mathematicians in the context of periodic crystals (Senechal 1979; Schwarzenberger 1984; Senechal 1988; Lifshitz 1997; De Las Pe–as et al. 1999; De Las Pe–as \& Felix 2007). Only a few mathematicians to date have considered the classification of color groups in the context of quasicrystals and non-periodic crystals, mostly using tools in Fourier analysis and algebraic number theory (Lifshitz 1997; Baake et al. 2002; Baake \& Grimm 2004). More recently, Bugarin et al. (2008) have introduced an approach in studying color symmetries of quasicrystals using notions on ideals. Then, Bugarin et al. (2009) have investigated perfect colorings of cyclotomic integers also using ideals and concepts in algebraic number theory. In this paper, we continue the work started by Bugarin et al. (2008 \& 2009) and this time, the color symmetries of planar modules is studied using an approach involving matrices. We explain the setting and basis for this method in the next section.

\section{\bf Bravais colorings of planar modules}

\subsection{\bf Setting}

In studying the color symmetries of $\M_n$, we consider colorings of $\M_n$ which are compatible with its underlying symmetry. A restriction to be imposed is that a color occupies a subset which is of the same Bravais type as the original set, while the other colors code the cosets. Assuming this compatibility requirement, a {\em Bravais coloring} $\mathcal{C}$ of $\M_n$ is arrived at by considering a coloring using cosets of a principal ideal $I$ of $\M_n$. Given $I$ of index $\ell$ in $M_n$, each element of $\M_n$ is assigned a color from a set of $\ell$ distinct colors. Two elements of $\M_n$ are assigned the same color if and only if they belong to the same coset of $I$. This coloring of the elements in $\M_n$ will be referred to as the {\em Bravais coloring} $\mathcal{C}$ of $\M_n$ determined by the ideal $I$.

A geometric representation of a Bravais coloring of $\M_n$ is achieved using a coloring of vertices of a discrete planar $N$-fold tiling ($N = n$ if $n$ is even, and $N = 2n$ if $n$ is odd). The tiling is periodic for $n\in\{3, 4\}$ and non-periodic for $n > 4$. For the latter, since $\M_n$ is dense on $\R^2$, only a subset of $\M_n$ is considered when coloring the discrete vertex set of an $N$-fold tiling on the plane. The entire dense module is seen as a lattice in a space of higher dimension, in particular, in dimension $\varphi(n)$.

In Figure \ref{fig1}(A), we display a Bravais coloring of $\M_5$ using the ideal $I = \langle1 - \xi\rangle$ as a vertex coloring of a 10-fold non-periodic tiling on the plane. This coloring depicts quasicrystallographic symmetries. In Figure \ref{fig1}(B), a Bravais coloring of $\M_7$ determined by the ideal $I = \langle1 - \xi - \xi^3\rangle$ is manifested as a vertex coloring of a 7-fold non-periodic tiling on the plane. 

\subsection{\bf Methodology}

There are three groups we look at in the analysis of a Bravais coloring of $\M_n$: the {\em symmetry group} $G$ of the uncolored module $\M_n$, the {\em color symmetry group} $H$ consisting of elements of $G$ which effect a permutation of the colors and the {\em color fixing group} $K$ consisting of elements of $H$ which fix the colors. If the color symmetry group corresponding to a Bravais coloring of $\M_n$ is $G$, then we say the coloring of $\M_n$ is perfect. Otherwise we have a non-perfect coloring of $\M_n$.

If $h\in H$ then $h$ determines a permutation of the colors $C$ in a Bravais coloring of $\M_n$ defined as follows: If $c_i$ is a color in $C$ and $h$ maps an element of $\M_n$ colored $c_i$ to an element colored $c_j$, let $hc_i = c_j$. This defines an action of $H$ on $C$, which induces a homomorphism $f$ from $H$ to the group of permutations of the set $C$ of colors of the elements of $\M_n$. The kernel of $f$ is $K$, a normal subgroup of $H$.

In this work, the problem of determining the groups $H$ and $K$ for a particular Bravais coloring of $\M_n$ is analyzed by considering $\M_n$ as a $\varphi(n)$ dimensional lattice $\Lambda$ in $\R^{\varphi(n)}$. The symmetry group $G$ of $\Lambda$ consists of the group of isometries in $\R^{\varphi(n)}$ that send the lattice to itself. $G$ is a $\varphi(n)$ dimensional crystallographic group which is a semi-direct product of its subgroup of translations $T(G) = \{t_y : \R^{\varphi(n)} \mapsto \R^{\varphi(n)}; t_y: x \mapsto x + y (y\in\Lambda)\} \cong \Z^{\varphi(n)}$ and its point group, $P(G) \cong D_N$, the dihedral group of order $2N$. $P(G)$ is generated by the following transformations from $\R^{\varphi(n)}$ to $\R^{\varphi(n)}$; namely the $N$-fold rotation $\phi_1$ and the reflection $\phi_2$.

The set $\{1, \xi, \xi^2, \ldots, \xi^{\varphi(n)-1}\}$ is taken to be a set of basis vectors for $\Lambda$ in $\R^{\varphi(n)}$. Relative to this set, each element in $\Lambda$ may be represented by its coordinate vector $[b_1, b_2, \ldots, b_{\varphi(n)}]^T$. The principal ideal $I$ is then viewed as a sublattice $L$ of $\Lambda$ (a subgroup of $\Lambda$ of finite index), and so is represented relative to the basis $\{1, \xi, \xi^2, \ldots, \xi^{\varphi(n)-1}\}$ of $\Lambda$ by a $\varphi(n)\times\varphi(n)$ square matrix $S$ whose columns form a generating set for $L$.

\section{\bf Results on the Color Symmetry Group $H$ and Color Fixing Group $K$} 

The assertions that follow facilitate our derivation of $H$ and $K$. In the results, relative to the basis vectors $\{1, \xi, \xi^2, \ldots, \xi^{\varphi(n)-1}\}$ of the $\varphi(n)$ dimensional lattice $\Lambda$ in $\R^{\varphi(n)}$, the sublattice $L$ of $\Lambda$ is represented by the matrix $S$, the generators $\phi_1$, $\phi_2$ of $D_N$ are represented by $\varphi(n)\times\varphi(n)$ matrices $A$ and $B$, respectively.

\subsection{\bf Determination of the color symmetry group $H$}

\begin{thm}\label{thm311}
Consider a Bravais coloring of $\M_n$ determined by an ideal $I$. Then $T(G) \leq H$  and $\phi_1 \in H$.
\end{thm}

\begin{proof}
Let $\Lambda$ be the lattice in $\R^{\varphi(n)}$ representing $\M_n$ and $L$ be the sublattice of $\Lambda$ representing the ideal $I$. Suppose $t\in T(G)$. Then if $x\in\Lambda$, $t(x) = x + y$, for some $y\in\Lambda$. This implies $t(x + L) = x + y + L\in\Lambda/L$. Hence $T(G) \leq H$. Moreover, $\phi_1(x) = y'$ where $y'\in\Lambda$ is the image of $x$ under the $N$-fold rotation $\phi_1$. Thus $\phi_1(x + L) = y' + L\in\Lambda/L$. Hence, $\phi_1\in H$.
\end{proof}

The above theorem implies that there are only two possibilities for the color symmetry group $H$, either $H = G$, that is, $H = T(G) \rtimes D_N$, where $D_N = \langle\phi_1, \phi_2\rangle$; or $H = T(G) \rtimes C_N$, where $C_N = \langle\phi_1\rangle$. This result also appears in the paper by Bugarin et al. (2008), verified by an approach using ideals. The following result determines when $\phi_2\in H$.

\begin{thm}\label{thm312}
Consider a Bravais coloring of $\M_n$ determined by $I$. Suppose $\Lambda$ is the lattice in $\R^{\varphi(n)}$ representing $\M_n$ and $L$ is the sublattice of $\Lambda$ representing $I$. Then $H = T \rtimes D_N$ if and only if $\phi_2(L)=L$.
\end{thm}

\begin{proof}
If $\phi_2\in H$, then $\phi_2$ permutes the cosets of $L$ in $\Lambda$. This implies $\phi_2(L) = L$ since $0\in L$ and $\phi_2(L)$ is the coset of $L$ containing $\phi_2(0) = 0$. Conversely, suppose $\phi_2(L) = L$. Let $x + L, x\in\Lambda$, be a coset of $L$ in $\Lambda$. Then $\phi_2(x +L) = \phi_2(x) + \phi_2(L) = \phi_2(x) + L\in\Lambda/L$ since $\phi_2\in P(G)$ and leaves $\Lambda$ invariant. Thus $\phi_2\in H$ and $H = T(G) \rtimes D_N$.
\end{proof}

\begin{cor}\label{cor313}
Let $S$ be a matrix representing the sublattice $L$ corresponding to the principal ideal $I$ of $\M_n$ and $B$ a matrix representing $\phi_2$. Then $H = T(G) \rtimes D_N$ if and only if $S^{-1}BS$ is an integral matrix.
\end{cor}

\begin{proof}
If $S$ is a matrix representing the sublattice $L$ relative to the basis $\{1, \xi, \xi^2, \ldots, \xi^{\varphi(n)-1}\}$, and $B$ is a matrix representing $\phi_2$, invariance of $L$ under $\phi_2$ is equivalent to the condition that $S^{-1}BS$ is an integral matrix, that is, all of its entries are integers.
\end{proof}

\begin{ex}\label{ex314}
Consider the ideal $I_1 = \langle1 - \xi_5\rangle$ in $\M_5$. The sublattice representing $I_1$ relative to the set of basis vectors $\{1, \xi_5, \xi_5^2, \xi_5^3 : \xi_5 = \exp(2\pi i/5)\}$ is given by the matrix $$S_5=\left[\begin{array}{cccc} 1 & 0 & 0 & 1 \\ -1 & 1 & 0 & 1 \\ 0 & -1 & 1 & 1 \\ 0 & 0 & -1 & 2 \end{array}\right];$$ a matrix representing the reflection $\phi_2\in D_{10}$ is $$B_5=\left[\begin{array}{cccc} 1 & -1 & 0 & 0 \\ 0 & -1 & 0 & 0 \\ 0 & -1 & 0 & 1 \\ 0 & -1 & 1 & 0 \end{array}\right].$$ Since $S_5^{-1}B_5S_5$ is an integral matrix, then $\phi_2\in H$. It follows that the Bravais coloring determined by $I_1$ in $\M_5$ is perfect, that is $H = T \rtimes D_{10}$. On the other hand, the Bravais coloring determined by the ideal $I_2 = \langle1 -\xi_7 -\xi_7^3\rangle$ in $\M_7$ is non-perfect. Note that $S_7^{-1}B_7S_7$ is non-integral, where a sublattice representing $I_2$ relative to the set of basis vectors $\{1, \xi_7, \xi_7^2, \xi_7^3, \xi_7^4, \xi_7^5, \xi_7^6 : \xi_7 = \exp(2\pi i/7)\}$ is given by $$S_7=\left[\begin{array}{cccccc} 1 & 0 & 0 & 1 & -1 & 1 \\ -1 & 1 & 0 & 1 & 0 & 0 \\ 0 & -1 & 1 & 1 & 0 & 1 \\ -1 & 0 & -1 & 2 & 0 & 1 \\ 0 & -1 & 0 & 0 & 1 & 1 \\ 0 & 0 & -1 & 1 & -1 & 2  \end{array}\right];$$ and the matrix representing $\phi_2'\in D_{14}$ is given by $$B_7=\left[\begin{array}{cccccc} 1 & -1 & 0 & 0 & 0 & 0 \\ 0 & -1 & 0 & 0 & 0 & 0 \\ 0 & -1 & 0 & 0 & 0 & 1 \\ 0 & -1 & 0 & 0 & 1 & 0 \\ 0 & -1 & 0 & 1 & 0 & 0 \\ 0 & -1 & 1 & 0 & 0 & 0  \end{array}\right].$$ For this coloring, $H = T \rtimes C_{14}$. \hfill$\Box$
\end{ex}

\subsection{\bf Determination of the color fixing group $K$}

\begin{thm}\label{thm321}
Consider a Bravais coloring of $\M_n$ determined by $I$. Suppose $\Lambda$ is the lattice in $\R^{\varphi(n)}$ representing $\M_n$ and $L$ is the sublattice of $\Lambda$ representing $I$. Then $T(I)$ is in $K$, where $T(I) = \{t_s: \R^{\varphi(n)}\mapsto\R^{\varphi(n)}; t_s: z \mapsto z + s (s\in L)\}$.  
\end{thm}

\begin{proof}
Let $t_y\in T(G)$, where $y\in\Lambda$. Then $t_y\in K$ if and only if for every $x\in\Lambda, t_y(x + L) = x + y + L = x + L$ if and only if $y\in L$. This implies $K \cap T(G) = T(I)$. Hence, $T(I)$ is in $K$.  
\end{proof}

A result similar to Theorem \ref{thm321} also appears in the paper by Bugarin et al. (2008), verified by an approach using ideals. It follows therefore that $K$ is of the form $T(I) \rtimes P_N$, where $P_N$ is a subgroup of $D_N$. To determine $P_N$, we verify which of the elements of $D_N$ fix the cosets of $I$ in $\M_n$. The next result provides a method in arriving at $K$.

\begin{thm}\label{thm322}
Consider a Bravais coloring of $\M_n$ determined by $I$, where $I$ is of index $\ell$ in $\M_n$. Let $S$ be the matrix representing the sublattice $L$ corresponding to the principal ideal $I$ of $\M_n$; $B$, $A$ be matrices representing respectively, $\phi_2$, $\phi_1$ and $x_i + L$ be the cosets of $L$ in $\Lambda$, for $i=1,\ldots,\ell$. Then the following are true:
\begin{enumerate}
\item[(i)] $\phi_2\in K$ if and only if $\phi_2\in H$ and $S^{-1}(Bx_i-x_i)$ is an integral matrix for all $i = 1, \ldots, \ell$; and,
\item[(ii)] $\phi_1^k\in K$, where $k$ divides $N$, if and only if $S^{-1}(A^kx_i-x_i)$ is an integral matrix for all $i = 1, \ldots, \ell$.
\end{enumerate}
\end{thm}

\begin{proof}
$\phi_2\in K$ if and only if $\phi_2\in H$ and $\phi_2$ fixes the cosets of $L$ in $\Lambda$, that is, $\phi_2(x_i + L)=x_i + L, i=1,\ldots,\ell$. Now, $\phi_2$ fixes the cosets if and only if for a matrix $B$ representing $\phi_2$, $Bx_i = x_i + Sy$ for some $y\in\Lambda, i=1,\ldots,\ell$. This is equivalent to the condition that $S^{-1}(Bx_i-x_i)$ is an integral matrix. The same argument follows for (ii).
\end{proof}

To illustrate the above ideas let us recall the examples discussed in Illustration \ref{ex314}.

\begin{ex}\label{ex322}
Let us calculate $K$ for the Bravais coloring of $\M_5$ determined by $I_1 = \langle1 -\xi_5\rangle$. The determinant of $S_5$ is 5, which is also the index of $I_1$ in $\M_5$. We obtain five cosets of $I_1$ in $\M_5$. A set of matrices representing the cosets of the sublattice representing $I_1$ is given by $S_5, [1, 0, 0, 0]^T + S_5, [2, 0, 0, 0]^T + S_5, [3, 0, 0, 0]^T + S_5$ and $[4, 0, 0, 0]^T+S_5$. Since the coloring determined by $I_1$ is perfect, we check which of the elements of $D_{10}$ fix all the cosets of $I_1$. This is equivalent to verifying whether $S_5^{-1}(R^k [a, 0, 0, 0]^T-[a, 0, 0, 0]^T)$ is an integral matrix for $a\in\{0, 1, 2, 3, 4\}$ where $R^k$ is a matrix representing an element $r$ in $D_{10}$ using the basis vectors $\{1, \xi_5, \xi_5^2, \xi_5^3 : \xi_5 = \exp(2\pi i/5)\}$. We have the following computations:
\end{ex}

\subsection{\bf Further results on the values of $H$ and $K$.}

Bugarin et al. (2008) enumerated the values of $H$ and $K$ for the cases $n\in\{5, 8, 12\}$. Then, additional values are provided for the cases $n\in\{3, 4, 7, 9\}$ by Bugarin et al. (2009). To this end, we provide the calculated values for the cases $n\in\{15, 16\}$ using the methods presented above. See Tables \ref{tab1} and \ref{tab2}.

\section{\bf Magnetic point groups and Bravais colorings}

In this part of the paper, we discuss magnetic point groups, also known as anti-symmetry groups, or black-and-white groups and give examples of how a Bravais coloring of a planar module may be used to describe the type of magnetic point group a crystal or a quasicrystal possesses. In particular, we will associate magnetic point groups with Bravais colorings having two colors. 

Magnetic point groups play an important role in the description of discrete point sets in which the points are not only characterized by their spatial coordinates but also by an additional property taking one of two possible values which will be assigned color black or white, or, spin ``up'' or ``down''. A single external operation of order $2$ is introduced (denoted by $e'$), which interchanges the colors or the spins everywhere throughout the point sets.

A $d$-dimensional magnetic point group is a subgroup of $O(d) \times 1'$, where $O(d)$ is the group of $d$-dimensional rotations (including improper rotations such as mirror reflections) and $1'$ is the $1$-dimensional spin space consisting of two values $e$ and $e'$, for which $e$ is the identity and fixes the spin configuration (or colors), and $e'$ is the time inversion that interchanges the two possible values (the colors or the spins). A magnetic point group $G^*$ formed from a subgroup $S^*$ of $O(d)$ is of three types, which we enumerate as follows:

\begin{enumerate}
\item The {\em grey} group is denoted by $G^* = S^* \times 1'$, where all rotations in $S^*$ occur with and without time inversion.
\item The {\em black-and-white} group is given by $G^* = T^* + (\phi T^*)e'$  where $T^*$ is an index $2$ subgroup of $S^*$. In particular, $S^* = T^* + \phi T^*, \phi \notin T^*$. In this group, half of the rotations in $S^*$, those belonging to $\phi T^*$, occur with time inversion, and the rotations in $T^*$ do not.
\item The {\em white} group pertains to $G^* = S^*$. In this case, the magnetic point group contains rotations that are not followed by time inversion, that is, all the colors or spins are preserved.
\end{enumerate}

Any ordinary point group $P$ associated with a crystal is itself a magnetic point group of type 3 ({\em white} group). If we apply $1'$ to each of the elements of $P$, a magnetic point group of type 1 ({\em grey group}) is formed. To arrive at a magnetic point group of type 2 ({\em black-and-white} group), one has to list all distinct subgroups $L$ of index 2 in $P$. Each element of $P$ that is in $L$ is unprimed in $G^*$ and the rest of the elements in $P$ are primed in $G^*$. To denote magnetic point groups of type 2, we follow the Hermann-Maugin symbol for $P$ and add prime to those elements of $P$ not in $L$. 

\subsection{\bf Magnetic point groups of crystals}

The magnetic point group of a $d$-dimensional magnetically-ordered periodic crystal containing an equal number of black and white colors or ``up'' and ``down'' spins, is the set of rotations from $O(d) \times 1'$ that leave the crystal invariant to within a translation. The magnetic point group $G^*$ of a crystal can either be a {\em grey} group or a {\em black-and-white group}. $G^*$ is a grey group if the time inversion, by itself, leaves the crystal invariant to within a translation. If time inversion together with a translation cannot leave the crystal invariant, then $G^*$ is a black-and-white group. The following illustration makes this apparent.

In Figure \ref{fig2}(A), we present a square crystal whose magnetic point group is a {\em grey} group. Its ordinary point group is given by $4mm$. The time inversion $e'$ together with a translation leaves the crystal invariant, thus its magnetic point group is $4mm1'$.  

Figure \ref{fig2}(B) on the other hand shows an example of a hexagonal crystal whose magnetic point group is a {\em black-and-white} group. Note that applying a translation combined with $e'$ cannot recover the original crystal and thus its magnetic point group cannot be a grey group. Instead its magnetic point group is $6'mm'$, which is formed from its ordinary point group $6mm$. To recover the original crystal, $e'$ must be combined with a point group operation. The six-fold rotation with center at the center of a hexagon, together with $e'$, leaves the crystal invariant. In like manner, the reflection with horizontal axis passing through the center of a hexagon, followed by $e'$, also leaves the crystal invariant. This explains the occurence of the ``prime'' ($'$) in $6$ and the second $m$ in $6'mm'$. Observe that a reflection with vertical axis passing through a center of the hexagon leaves the crystal invariant without the time inversion. This reflection is represented by the first $m$ in $6'mm'$ and thus is unprimed in the symbol.

In terms of a Bravais coloring of a module, the magnetically-ordered square crystal given in Figure \ref{fig2}(A)  can be represented by a Bravais coloring of $\M_4$ (square lattice) determined by the principal ideal  $\langle 1 + i\rangle$. To describe the hexagonal crystal shown in Figure \ref{fig2}(B), since there is no Bravais coloring of $\M_3 = \M6$ corresponding to an ideal of index two in $\M_3$, we consider a perfect two-coloring of the faces of the tiling by equilateral triangles with a color fixing group isomorphic to the group generated by a 3-fold rotation about the center of a hexagon and two linearly independent translations. 

\subsection{\bf Magnetic point groups of quasicrystals}

The magnetic point group of a $d$-dimensional magnetically ordered quasiperiodic crystal, or simply quasicrystal, is defined as the set of rotations from $O(d) \times 1'$ that leave the crystal indistinguishable (Lifshitz 2005). This means that the original crystal as well as the rotated or reflected crystal contain the same spatial distributions of finite numbers of black or white colors of arbitrary size. The two are statistically the same though not necessarily identical (Lifshitz 1997). For the case of periodic crystals, the requirement of indistinguishability reduces to the requirement of invariance to within a translation.

An example of a quasicrystal with a {\em grey} group as its magnetic point group is shown in Figure \ref{fig3}(A). The  octagonal quasicrystal has magnetic point group $8mm1'$ formed from its ordinary point group $8mm$.  The time inversion rearranges the colors or spin clusters in the quasicrystal. A spin cluster, (denoted by one color) and its image under inversion appear in the quasicrystal with the same spatial distribution. 

Figure \ref{fig3}(B) shows a decagonal quasicrystal with a black and white group as its magnetic point group. More particularly, the quasicrystalÕs magnetic point group is $10'm'm$, formed from its ordinary point group $10mm$. In this case, the time inversion by itself, does not leave the quasicrystal indistinguishable. The inversion must be combined either with odd powers of a $10$-fold rotation, or with reflections along a vertical line. However, a reflection along a horizontal line by itself leaves the quasicrystal indistinguishable, and hence is unprimed in the symbol.

The octagonal quasicrystal shown in Figure \ref{fig3}(A) can be realized as a Bravais coloring of the module $\M_8$ determined by $\langle 1 + \xi_8\rangle$, via the vertices of the Ammann-Beenker tiling. Since there is no Bravais coloring of $\M_5$ corresponding to an ideal of index two in $\M_5$, the tiles of the decagonal lattice (in this case, a Penrose tiling) are colored to produce a two-coloring shown in Figure \ref{fig3}(B) that describes the decagonal quasicrystal with a {\em black-and-white} magnetic point group.

We end this section with the following remark.

\noindent{\bf Remark:} The magnetic point group of a $\varphi(n)$-dimensional magnetically ordered periodic or quasiperiodic crystal is a {\em grey} group if the corresponding module $\M_n$ has a Bravais coloring induced by an ideal of index two. Otherwise, the crystal has a {\em black-and-white} magnetic point group. 

The reader may refer to Baake \& Grimm (2004) for a list of the number of ideals of a given finite index in $\M_n$.

\section{\bf Summary}

In this work, we have presented a means to study the color symmetries of periodic and non-periodic structures (including quasicrystallographic structures) via Bravais colorings of planar modules that emerge as the rings of integers in cyclotomic fields with class number one. In particular, the contribution of this work is to provide a method that derives the color symmetry groups and color fixing groups associated with these modules using an approach involving matrices. The derivation centers on the elements of the module treated as points of a lattice in higher dimensional space. The method, implemented with the aid of a computer algebra system facilitates a systematic way to arrive at the results on color symmetry groups and color fixing groups of modules other than those already provided by Bugarin et al. (2008 \& 2009) using an approach involving ideals. The algorithm facilitates the calculation for the modules $M_{15}$ and $\M_{16}$ and provides a springboard to address the characterization of the color fixing group for the other symmetry cases grouped into classes of equal value of $\varphi(n)$.

This paper also provides instances where the link between Bravais colorings of planar modules to particular magnetic point groups of crystals and quasicrystals is exhibited. This connection may become more apparent if one considers a generalization of the study of magnetic point groups to involve colorings of more than two colors.

\begin{figure}
  \centering
  \subfloat[]{\includegraphics[width=0.4\textwidth]{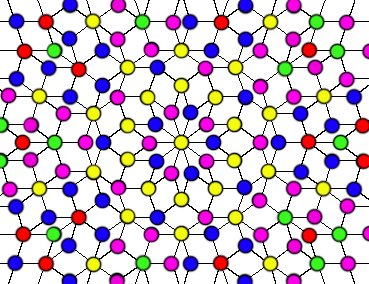}} \,\,\,\,\,                
  \subfloat[]{\includegraphics[width=0.4\textwidth]{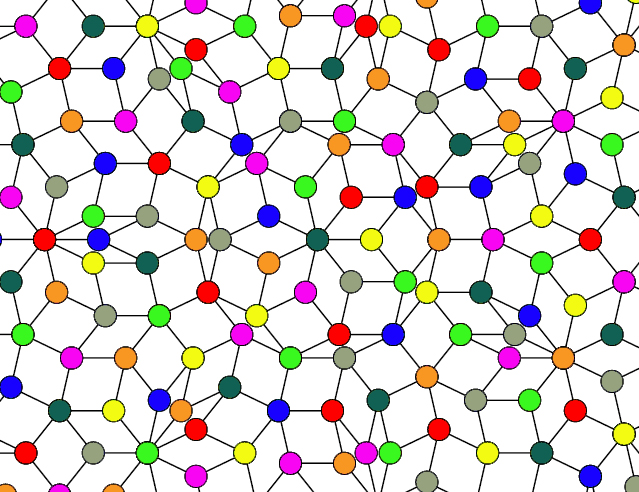}}
  \caption{Bravais colorings of (A) $\M_5$ determined by $I_1 = \langle1 -\xi_5\rangle$; and (B) $\M_7$ determined by $I_2  = \langle1 - \xi_7 - \xi_7^3\rangle$.}
  \label{fig1}
\end{figure}

\begin{figure}
  \centering
  \subfloat[]{\includegraphics[width=0.4\textwidth]{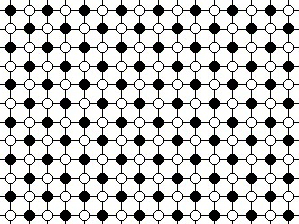}} \,\,\,\,\,                
  \subfloat[]{\includegraphics[width=0.4\textwidth]{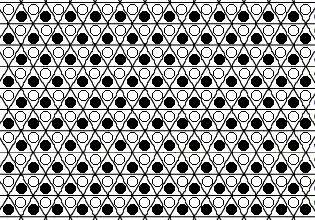}}
  \caption{Two-colorings exhibiting magnetic point group symmetries of crystals. (Adapted from Lifshitz 2005).}
  \label{fig2}
\end{figure}

\begin{figure}
  \centering
  \subfloat[]{\includegraphics[width=0.4\textwidth]{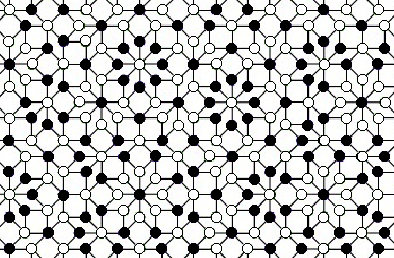}} \,\,\,\,\,                
  \subfloat[]{\includegraphics[width=0.4\textwidth]{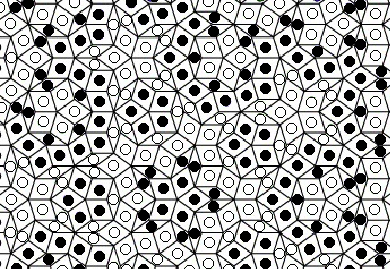}}
  \caption{Two-colorings exhibiting magnetic point group symmetries of quasicrystals. (Adapted from Lifshitz 2005).}
  \label{fig3}
\end{figure}

\begin{table}
\begin{tabular}{cccccc}
$n$ & $\ell$ & $H$ & $K$ & $I$ \\ \hline
15 & 16 & $T(G) \rtimes C_{30}$ & $T(I) \rtimes C_2$ & $1+\xi_{15}+\xi_{15}^4$ \\
 & 25 & $T(G) \rtimes D_{30}$& $T(I)\rtimes C_5$ & $1-\xi_{15}^3$ \\
 & 31 & $T(G) \rtimes C_{30}$ & $T(I)$ & $1+\xi_{15}+\xi_{15}^3$\\
 & 61 & $T(G) \rtimes C_{30}$ & $T(I)$ & $1+\xi_{15}^3+\xi_{15}^5+\xi_{15}^7$\\
 & 81 & $T(G) \rtimes D_{30}$ & $T(I)\rtimes C_3$ & $1-\xi_{15}^5$\\
 & 121 & $T(G) \rtimes C_{30}$ & $T(I)$ & $2+\xi_{15}^3+\xi_{15}^6$\\
 & 151 & $T(G) \rtimes C_{30}$ & $T(I)$ & $1-2\xi_{15}$\\
 & 181 & $T(G) \rtimes C_{30}$ & $T(I)$ & $1+\xi_{15}+\xi_{15}^3+\xi_{15}^5+\xi_{15}^7$\\
 & 211 & $T(G) \rtimes C_{30}$ & $T(I)$ & $1+\xi_{15}^2+2\xi_{15}^3$\\
 & 241 & $T(G) \rtimes C_{30}$ & $T(I)$ & $1-\xi_{15}^6+2\xi_{15}^7-\xi_{15}^9$\\
& 256' & $T(G) \rtimes D_{30}$ & $T(I) \rtimes C_2$ & $2$ \\
& 256' & $T(G) \rtimes C_{30}$ & $T(I)$ & $(1+\xi_{15}+\xi_{15}^4)^2$ \\
& $> 256$ & * & $T(I)$ & * \\ \hline  \end{tabular}
\vspace{2mm}
\caption{\label{tab1} The values of $H$ and $K$ for the case $n = 15$. Notice that there are two cases for which the Bravais coloring exhibits $256$ colors, the one being perfect and the other being non-perfect. Entries with asterisk (*) mean that there are several particular cases.}
\end{table}

\begin{table}
\begin{tabular}{cccccc}
$n$ & $\ell$ & $H$ & $K$ & $I$ \\ \hline
16 & 2  & $T(G) \rtimes D_{16}$ & $T(I)\rtimes D_{16}$ & $1-\xi_{16}$ \\
  & 4  & $T(G) \rtimes D_{16}$ & $T(I)\rtimes D_{8}$ & $1-\xi_{16}^2$ \\
 & 8  & $T(G) \rtimes D_{16}$ & $T(I)\rtimes C_{4}$ & $(1-\xi_{16})^3$ \\ 
  & 16  & $T(G) \rtimes D_{16}$ & $T(I)\rtimes C_{4}$ & $1-\xi_{16}^4$ \\
   & 17  & $T(G) \rtimes C_{16}$ & $T(I)$ & $1-\xi_{16}+\xi_{16}^3$\\
  & 32  & $T(G) \rtimes D_{16}$ & $T(I)\rtimes C_{2}$ & $(1-\xi_{16})^5$ \\
   & 34  & $T(G) \rtimes C_{16}$ & $T(I)$ & $(1-\xi_{16})(1-\xi_{16}+\xi_{16}^3)$ \\
   & 49  & $T(G) \rtimes C_{16}$ & $T(I)$ & $1-\xi_{16}-\xi_{16}^2$ \\
   & 64  & $T(G) \rtimes D_{16}$ & $T(I)\rtimes C_{2}$ & $(1-\xi_{16}^2)^3$ \\
   & 68  & $T(G) \rtimes C_{16}$ & $T(I)$ & $(1-\xi_{16}^2)(1-\xi_{16}+\xi_{16}^3)$ \\
   & 81  & $T(G) \rtimes C_{16}$ & $T(I)$ & $1+\xi_{16}^4+\xi_{16}^6$ \\
   & 97  & $T(G) \rtimes C_{16}$& $T(I)$ & $1+2\xi_{16}^3+\xi_{16}^5+\xi_{16}^7$ \\
   & 98  & $T(G) \rtimes C_{16}$ & $T(I)$ & $(1-\xi_{16})(1-\xi_{16}-\xi_{16}^2)$ \\
   & 113  & $T(G) \rtimes C_{16}$ & $T(I)$ & $2-2\xi_{16}+\xi_{16}^5$ \\
   & 128  & $T(G) \rtimes D_{16}$ & $T(I)\rtimes C_{2}$ & $(1-\xi_{16})^7$ \\
   & 136  & $T(G) \rtimes C_{16}$ & $T(I)$ & $(1-\xi_{16})^3(1-\xi_{16}+\xi_{16}^3)$ \\
   & 162  & $T(G) \rtimes C_{16}$ & $T(I)$ & $(1-\xi_{16})(1+\xi_{16}^4+\xi_{16}^6)$ \\
   & 193  & $T(G) \rtimes C_{16}$ & $T(I)$ & $1-\xi_{16}^2-\xi_{16}^3+\xi_{16}^4+\xi_{16}^7$ \\
   & 194  & $T(G) \rtimes C_{16}$ & $T(I)$ & $(1-\xi_{16})(1+2\xi_{16}^3+\xi_{16}^5+\xi_{16}^7)$ \\
   & 196  & $T(G) \rtimes C_{16}$ & $T(I)$ &$(1-\xi_{16}^2)(1-\xi_{16}-\xi_{16}^2)$ \\
   & 226  & $T(G) \rtimes C_{16}$ & $T(I)$ & $(1-\xi_{16})(2-2\xi_{16}+\xi_{16}^5)$ \\ 
   & 241  & $T(G) \rtimes C_{16}$ & $T(I)$ & $1-\xi_{16}-\xi_{16}^2+\xi_{16}^3+\xi_{16}^5$ \\
   & 256  & $T(G) \rtimes D_{16}$ & $T(I) \rtimes C_2$ & $2$ \\
   & $> 256$ & *  & $T(I)$ & * \\ \hline \end{tabular}
\vspace{2mm}
\caption{\label{tab2} The values of $H$ and $K$ for the case n = 16. Entries with asterisk (*) mean that there are several particular cases.}
\end{table}

\section*{Acknowledgement}

The authors are grateful to Rene Felix of the Institute of Mathematics, University of the Philippines for helpful discussions. Ma. Louise De Las Pe\~nas would like to acknowledge funding support of the Ateneo de Manila University through the Loyola Schools Scholarly Work Faculty Grant. Enrico Paolo Bugarin thanks the German Collaborative Research Council (CRC 701) for a support during his stay in Bielefeld University, Germany in 2009, where much of the computations in this work were done.

\end{document}